\documentclass[12pt,reqno]{amsart}
\theoremstyle{plain}
\newtheorem{theorem}{Theorem}[]
\newtheorem{lemma}{Lemma}[]

\newtheorem{example}{Example}[]
\newtheorem{remark}{Remark}[]
\usepackage{amssymb,bm,graphicx,graphics,mathrsfs,bbm,color}
\allowdisplaybreaks[4]
\begin{document}
\def\R{\mathbb{R}}
\def\P{\mathbb{P}}
\def\hK{h_{\scriptscriptstyle K}}
\def\oK{\omega_{\scriptscriptstyle K}}
\def\om{{\omega}}
\def\Com{C_\dag}
\def\Cinv{{C_\sharp}}
\def\CoT{{C_\dag}}
\def\CInt{{C_\flat}}
\def\Cdh{C_\diamond}
\def\diff{(\om-\oK)}
\def\p{\partial}
\def\Wi#1{{W^{#1}_\infty}}
\def\Ws#1{{W^{#1}_s}}
\def\cP{\mathcal{P}}
\def\cN{\mathcal{N}}
\def\PK{\Pi_{\scriptscriptstyle K}}
\title{A General Superapproximation Result}
\author{Susanne C. Brenner}
\address{Susanne C. Brenner, Department of Mathematics and Center for
 Computation and Technology, Louisiana State University, Baton Rouge, LA 70803, USA}
\email{brenner@math.lsu.edu}
\thanks{This work was supported in part
 by the National Science Foundation under Grant No.
 DMS-19-13035.}
\begin{abstract}
  A general superapproximation result is derived in this paper which is
  useful for the local/interior error analysis of finite element methods.
\end{abstract}
\maketitle
%
 Superapproximation results are useful tools for the local/interior
 error analysis of
 finite element methods \cite{Wahlbin:2000:Handbook}.
 Our goal is to derive a general superapproximation result
 that includes many of the results in the literature as special cases.
\par
 We will adopt the usual notation of function spaces and norms
  that can be found for example in
  \cite{ADAMS:2003:Sobolev,Ciarlet:1978:FEM,BScott:2008:FEM}.
  The set-up is as follows.
\begin{itemize}
\item $\om$ is a  $C^\infty$ function on $\R^N$ such that, for a positive number $d$,
\begin{equation}\label{eq:omBdd}
  |\om|_{W^j_\infty(\R^N)}\leq \Com d^{-j}  \qquad j=0,1,2,\ldots.
\end{equation}
\item $(K,\cP,\cN)$ is a finite element space ({\`a la} Ciarlet), where $K$ is a
 connected compact subset of $\R^N$,
 $\cP$ is a finite dimensional vector space of polynomials,
 $\cN$ is a set of nodal variables (degrees of freedom), and the diameter
 $\hK$ of $K$ satisfies
\begin{equation}\label{eq:dh}
 \hK\leq  d.
\end{equation}
 We will allow a slight abuse of notation to write the Sobolev spaces $H^\ell(\textrm{int}(K))$
 (resp.,  $W^\ell_\infty(\textrm{int}(K))$) as
 $H^\ell(K)$ (resp.,  $W^\ell_\infty(K)$).
\item $\oK$ is the mean of $\om$ over $K$ so that
\begin{align}
   |\oK|&\leq \Com,\label{eq:oTBdd}\\
   |\om-\oK|_{L_\infty(K)}&\leq \CoT d^{-1}\hK.\label{eq:diffBdd}
\end{align}
\item The linear operator $\PK:C^\infty(K)\rightarrow \cP$ is the nodal
interpolation operator that satisfies
\begin{equation}\label{eq:Invariant}
  \zeta-\PK\zeta=0\qquad\forall\,\zeta\in\cP,
\end{equation}
 and there exists a positive integer $k$
(depending on the maximum order of differentiation among the
 nodal variables in $\cN$) such that
\begin{equation}\label{eq:IntEst}
  |\zeta- \PK\zeta|_{\Wi{\ell}(K)}\leq \CInt \hK^{k+1-\ell}|\zeta|_{\Wi{k+1}(K)}
  \qquad\forall\,\zeta\in C^\infty(K),\;0\leq\ell\leq k+1.
\end{equation}
\item
 The following  inverse estimates are satisfied:
\begin{alignat}{3}
  |\chi|_{\Ws{q}(K)}&\leq \Cinv \hK^{p-q}|\chi|_{\Ws{p}(K)}&\qquad&
   \forall\,\chi\in\cP,\;s=2,\infty, \;
   \text{and}\;0\leq p<q\leq k+1,\label{eq:InvEst}\\
   |\chi|_{\Wi{p}(K)}&\leq C_\diamond \hK^{-N/2}|\chi|_{H^p(K)}
   &\qquad&\forall\,\chi\in\cP\;\text{and}\;0\leq p\leq k.\label{eq:InvEst2}
 \end{alignat}
\end{itemize}
\par
  Inverse estimates and interpolation error estimates for
  finite elements can be found in
  \cite{Ciarlet:1978:FEM,BScott:2008:FEM}.  Note that,
  by the Bramble-Hilbert lemma \cite{BH:1970:Lemma},
  the interpolation error estimates in \eqref{eq:IntEst} are valid
  provided the space $\cP$ of shape
  functions contains all the polynomials of total degree $\leq k$.
\par
 The following theorem is the main result of this paper.
\begin{theorem}\label{thm:SA} Given $0\leq m\leq k$, $0\leq \ell\leq m+1$ and $n\geq 1$,
 we have,
 under assumptions \eqref{eq:omBdd}--\eqref{eq:InvEst2},
\begin{alignat}{3}
 |\om^{n}\chi-\PK(\om^{n}\chi)|_{H^\ell(K)}
  &\leq C\hK^{m+1-\ell}d^{-(m+1)}\sum_{j=0}^{m} d^j |\chi|_{H^j(K)}
   &\qquad&\forall\,\chi\in\cP,
  \label{eq:SA1}
\intertext{and, for $n\geq m+1$,}
 |\om^{n}\chi-\PK(\om^{n}\chi)|_{H^\ell(K)}
  &\leq C\hK^{m+1-\ell}d^{-(m+1)}\sum_{j=0}^{m} d^j |\om^j\chi|_{H^j(K)}
   &\qquad&\forall\,\chi\in\cP,
  \label{eq:SA2}
\end{alignat}
 where
  the positive constant $C$ in \eqref{eq:SA1} and \eqref{eq:SA2} depends only on $N$,
 $n$, $k$, $\Com$, $\CInt$,  $\Cinv$  and $C_\diamond$.
\end{theorem}
\begin{remark}\label{rem:SA}\rm
  For an arbitrary smooth function $\chi$, the interpolation error estimates in
  \eqref{eq:SA1} and \eqref{eq:SA2} will involve the $(m+1)$-order Sobolev norm
  $|\chi|_{H^{m+1}(K)}$.  The fact that the right-hand sides of \eqref{eq:SA1} and \eqref{eq:SA2}
  only involve Sobolev norms up to order $m$ can be exploited to show that the
  energy norm of a discrete harmonic
  (or biharmonic) function on a subdomain can be bounded by a lower order norm on
  a larger subdomain, which is the discrete version of a Caccioppoli estimate.  Such estimates
  play a key role in the local/interior error analysis of finite element methods
  (cf. \cite[Section~9.1]{Wahlbin:2000:Handbook} and \cite[Section~3]{DGS:2011:Local}).
\end{remark}
\begin{remark}\label{rem:Observation}\rm
  The key observation is that, by the binomial theorem,
\begin{equation}\label{eq:Splitting}
  \om^n\chi-\PK(\om^n\chi)=\sum_{q=1}^n \binom{n}{q}\Big[\diff^q\oK^{n-q}\chi
  -\PK\big(\diff^q\oK^{n-q}\chi\big)\Big]\qquad\forall\,\chi\in\cP,
\end{equation}
 where the summation is over $q\geq 1$ because
  $\oK^n\chi-\PK(\oK^n\chi)=0$ by \eqref{eq:Invariant}.  This is
  why only derivatives up to order $m$ appear on the right-hand side
 of \eqref{eq:SA1} and \eqref{eq:SA2} (cf. Lemma~\ref{lem:Reduction1} below).
\end{remark}
\begin{example}\label{example:classical}\rm
  For $n=1$, $m=\ell=1$ or $m=\ell=0$, it follows from \eqref{eq:SA1}  that
\begin{align*}
  |\om\chi-\PK(\om\chi)|_{H^1(K)}&\leq C\hK d^{-2}
  \big(\|\chi\|_{L_2(K)}+d|\chi|_{H^1(K)}\big),\\
  \|\om\chi-\PK(\om\chi)\|_{L_2(K)}&\leq  C\hK d^{-1}\|\chi\|_{L_2(K)},
\end{align*}
 which cover the results in \cite[Section 2]{NS:1974:SA},
 \cite[Section~6]{BNS:1975:MaxNorm} and \cite[Section~3]{AL:1995:Stokes}.
\end{example}
\begin{example}\label{example:DG}\rm
  For $n=2$ and $m=\ell=1$, it follows from \eqref{eq:SA2}  that
\begin{equation*}
  |\om^2\chi-\PK(\om^2\chi)|_{H^1(K)}\leq
  C\hK d^{-2}\big(\|\chi\|_{L_2(K)}+d|\om\chi|_{H^1(K)}\big),
\end{equation*}
 which is the result in \cite[Section~2]{DGS:2011:Local}.
\end{example}
\begin{example}\label{example:c0IP}\rm
  For $n=4$, $m=2$ and $1\leq\ell \leq 3$, it follows from \eqref{eq:SA2} that
\begin{align*}
  |\om^4-\PK(\om^4\chi)|_{H^\ell(K)}&\leq C\hK^{3-\ell}d^{-3}
  \big(\|\chi\|_{L_2(K)}+d|\om\chi|_{H^1(K)}+d^2|\om^2\chi|_{H^2(K)}\big),
\end{align*}
 which is the result in \cite[Section~3.3]{Leykekhman:2019:C0IP}.
\end{example}
\par
  We will use the notation $A\lesssim B$ to represent the inequality
  $A\leq \text{(constant)}B$, where the hidden generic positive constant
  only depends on  $N$,
 $n$, $k$, $\Com$, $\CInt$,  $\Cinv$  and $C_\diamond$.
\par
 The proof of Theorem~\ref{thm:SA} requires two lemmas that are
 based on the product rule (Leibniz formula) and the binomial theorem.
 The first lemma allows the reduction of the
 order of the Sobolev seminorm.
\begin{lemma}\label{lem:Reduction1}
  For $1\leq q\leq n$, $0\leq p\leq k+1$ and $s=2,\infty$,  we have
\begin{equation}\label{eq:Reduction}
  d^{p}|\diff^q \chi|_{\Ws{p}(K)}\lesssim
  \sum_{j=0}^{(p-q)\vee0}d^j|\chi|_{\Ws{j}(K)}
  \qquad\forall\,\chi\in\cP,
\end{equation}
 where $a\vee b=\max(a,b)$.
\end{lemma}
\begin{proof}  Note that  \eqref{eq:omBdd}, \eqref{eq:dh},
 \eqref{eq:diffBdd} and the product rule imply
\begin{equation}\label{eq:R1}
  |\diff^q|_{\Wi{r}(K)}\lesssim
  \begin{cases}
    d^{-r}&\qquad\text{if $r> q$}\\[4pt]
   d^{-q}\hK^{q-r}&\qquad\text{if $r\leq q$}
  \end{cases}.
\end{equation}
\par
 Let $\chi\in\cP$ be arbitrary.
 It follows from \eqref{eq:R1} and the product rule again that
\begin{align}\label{eq:R2}
  &d^{p}|\diff^q \chi|_{\Ws{p}(K)}\lesssim d^{p}
  \sum_{r=0}^{p}|\diff^q|_{\Wi{r}(K)}|\chi|_{\Ws{p-r}(K)}
  \notag\\
   &\hspace{50pt}\lesssim
   \sum_{r=0}^{p\wedge q}d^{p-q}\hK^{q-r}|\chi|_{\Ws{p-r}(K)}+
   \sum_{r=(p\wedge q)+1}^{p} d^{p-r}|\chi|_{\Ws{p-r}(K)},
\end{align}
 where $p\wedge q=\min(p,q)$ and we have adopted the convention that a sum is void if
 the lower index is strictly greater than the upper index.
 Note that the second sum on the right-hand side of \eqref{eq:R2} is void
 if and only if $p\leq q$.
\par
 In the case where $p\leq q$, the first sum on the right-hand side of \eqref{eq:R2}
 satisfies
\begin{align*}
  \sum_{r=0}^{p\wedge q}d^{p-q}\hK^{q-r}|\chi|_{\Ws{p-r}(K)}&=\sum_{r=0}^{p}
    (\hK/d)^{q-p}\hK^{p-r}|\chi|_{\Ws{p-r}(K)}
    \lesssim \|\chi\|_{L_s(K)}\\
\end{align*}
 because of \eqref{eq:dh} and \eqref{eq:InvEst}.  On the other hand, if $p>q$, then we have
\begin{align*}
  \sum_{r=0}^{p\wedge q}d^{p-q}\hK^{q-r}|\chi|_{\Ws{p-r}(K)}&=\sum_{r=0}^{q}
  d^{p-q}\hK^{q-r}|\chi|_{\Ws{p-r}(K)}\lesssim d^{p-q}|\chi|_{\Ws{p-q}(K)}
\end{align*}
 because of \eqref{eq:InvEst}.
\par
 It follows that
\begin{equation*}
  \sum_{r=0}^{p\wedge q}d^{p-q}\hK^{q-r}|\chi|_{\Ws{p-r}(K)}\lesssim
  d^{p-(p\wedge q)}|\chi|_{\Ws{p-(p\wedge q)}(K)},
\end{equation*}
 which together with \eqref{eq:R2} implies
\begin{align*}
   d^p|\diff^q \chi|_{\Ws{p}(K)}&\lesssim
   \sum_{r=p\wedge q}^p d^{p-r}|\chi|_{\Ws{p-r}(K)}.
\end{align*}
 The estimate \eqref{eq:Reduction} is then obtained by the change of index $j=p-r$.
\end{proof}
\par
 The second lemma allows the switching from $|\oK^j\chi|_{H^j(K)}$ to
 $|\om^j\chi|_{H^j(K)}$.
\begin{lemma}\label{lem:Binomial}
 For $0\leq j\leq m$, we have
\begin{alignat}{3}\label{eq:Binomial}
  d^{j}|\oK^j\chi|_{H^j(K)}
  &\lesssim \sum_{p=0}^j d^p|\om^{p}\chi|_{H^p(K)}
  & \qquad&\forall\,\chi\in\cP.
\end{alignat}
\end{lemma}
\begin{proof}  According to the binomial theorem, we have
\begin{equation}\label{eq:B1}
\oK^j=\om^j-\sum_{\ell=1}^{j}\binom{j}{\ell}\diff^\ell\oK^{j-\ell}.
\end{equation}
 It follows from \eqref{eq:oTBdd}, Lemma~\ref{lem:Reduction1}
 (with $s=2$) and \eqref{eq:B1} that
\begin{align*}
  d^j|\oK^j\chi|_{H^j(K)}&\lesssim
  d^j|\om^j\chi|_{H^j(K)}+\sum_{\ell=1}^j|\oK|^{j-\ell}d^j|\diff^\ell\chi|_{H^j(K)}\\
  &\lesssim  d^j|\om^j\chi|_{H^j(K)}+
  \sum_{\ell=1}^j\sum_{p=0}^{j-\ell} d^p|\oK^{j-\ell}\chi|_{H^p(K)}\\
  &\lesssim d^j|\om^j\chi|_{H^j(K)}
    +\sum_{\ell=1}^j\sum_{p=0}^{j-\ell}d^p|\oK^p\chi|_{H^p(K)}\\
    &=d^j|\om^j\chi|_{H^j(K)}+\sum_{p=0}^{j-1}(j-p)d^p|\oK^p\chi|_{H^p(K)}\\
    &\lesssim d^j|\om^j\chi|_{H^j(K)}+\sum_{p=0}^{j-1}d^p|\oK^p\chi|_{H^p(K)},
\end{align*}
 which implies \eqref{eq:Binomial} through a recursive argument.
\end{proof}
 We are now ready to prove Theorem~\ref{thm:SA}.
 From \eqref{eq:IntEst} and Lemma~\ref{lem:Reduction1} (with $s=\infty$),
  we find, for $n\geq q\geq 1$,
\begin{align}\label{eq:P1}
  &|\diff^q\oK^{n-q}\chi
  -\PK\big(\diff^q\oK^{n-q}\chi\big)|_{H^\ell(K)}\notag\\
  &\hspace{80pt}\lesssim
  (\hK^{N/2}|\oK^{n-q}|)|\diff^q\chi
  -\PK\big(\diff^q\chi\big)|_{\Wi{\ell}(K)}\notag\\
  &\hspace{80pt}
  \lesssim (\hK^{N/2}|\oK^{n-q}|)\hK^{k+1-\ell}|\diff^q\chi|_{\Wi{k+1}(K)}\notag\\
   &\hspace{80pt}\lesssim (\hK^{N/2}|\oK^{n-q}|)\Big(
    \hK^{k+1-\ell}d^{-(k+1)}
    \sum_{j=0}^{(k+1-q)\vee 0} d^{j}|\chi|_{\Wi{j}(K)}\Big).
\end{align}
\par
 The next step is to reduce the summation on the right-hand side of \eqref{eq:P1} from the range
 $0\leq j\leq [(k+1-q)\vee0]$ to the range $0\leq j\leq [(m+1-q)\vee0]$.
\par
 In the case where $k+1-q\leq0$, we also have $m+1-q\leq 0$ (since $m\leq k$ is
 one of the assumptions in
 Theorem~\ref{thm:SA}) and hence, in view of \eqref{eq:dh},
\begin{align}\label{eq:C1}
   \hK^{k+1-\ell} d^{-(k+1)}
    \sum_{j=0}^{(k+1-q)\vee 0} d^{j}|\chi|_{\Wi{j}(K)}&=(\hK/d)^{k-m}
     \hK^{m+1-\ell} d^{-(m+1)}\|\chi\|_{L_\infty(K)}\notag\\
     &\leq \hK^{m+1-\ell}d^{-(m+1)}
    \sum_{j=0}^{(m+1-q)\vee 0} d^{j}|\chi|_{\Wi{j}(K)}.
\end{align}
\par
 Suppose $k+1-q>0$ and let $j$
  be an index within the range of the summation on the
  right-hand side of \eqref{eq:P1}, i.e.,
\begin{equation}\label{eq:Range}
  0\leq j\leq k+1-q.
\end{equation}
 If $j$ is strictly greater than  $m+1-q$, we have, by \eqref{eq:dh}, \eqref{eq:InvEst}
  and \eqref{eq:Range},
\begin{align}\label{eq:C2}
 \hK^{k+1-\ell} d^{-(k+1)}\big(d^j|\chi|_{\Wi{j}(K)}\big)&\lesssim
    \hK^{k+1-\ell}d^{-(k+1)}\big(d^j \hK^{m+1-q-j}|\chi|_{\Wi{m+1-q}(K)}\big)\notag\\
    &=(\hK/d)^{k+1-q-j}\hK^{m+1-\ell}d^{-(m+1)} \big(d^{m+1-q}|\chi|_{\Wi{m+1-q}(K)}\big)\notag\\
    &\leq \hK^{m+1-\ell}d^{-(m+1)}\big( d^{m+1-q}|\chi|_{\Wi{m+1-q}(K)}\big).
\end{align}
 If $j$ is less than or equal to $m+1-q$, we have
\begin{align}\label{eq:C3}
  \hK^{k+1-\ell}d^{-(k+1)}\big(d^j|\chi|_{\Wi{j}(K)}\big)
  &=(\hK/d)^{k-m}\hK^{m+1-\ell}d^{-(m+1)}
   \big(d^j|\chi|_{\Wi{j}(K)}\big)\notag\\
   &\leq \hK^{m+1-\ell}d^{-(m+1)}\big(d^j|\chi|_{\Wi{j}(K)}\big),
\end{align}
 again because $m\leq k$ is one of the assumptions in Theorem~\ref{thm:SA}.
\par
 Since all the possibilities are  covered by \eqref{eq:C1}, \eqref{eq:C2} and \eqref{eq:C3},
 we  conclude from \eqref{eq:InvEst2} and \eqref{eq:P1} that
\begin{align}\label{eq:P2}
  &|\diff^q\oK^{n-q}\chi
  -\PK\big(\diff^q\oK^{n-q}\chi\big)|_{H^\ell(K)}\notag\\
  &\hspace{80pt}\lesssim
   (\hK^{N/2}|\oK^{n-q}|)\Big( \hK^{m+1-\ell}d^{-(m+1)}
    \sum_{j=0}^{(m+1-q)\vee 0} d^{j}|\chi|_{\Wi{j}(K)}\Big)\notag\\
  &\hspace{80pt}\lesssim
   \hK^{m+1-\ell}d^{-(m+1)}
    \sum_{j=0}^{(m+1-q)\vee 0} d^{j}|\oK^{n-q}\chi|_{H^j(K)}.
\end{align}
\par
  Note that $n-q\geq0$ and hence, in view of \eqref{eq:oTBdd}, we
  can simply drop $\oK^{n-q}$ from the right-hand side of \eqref{eq:P2}
   to arrive at
\begin{align}\label{eq:P3}
   &|\diff^q\oK^{n-q}\chi
  -\PK\big(\diff^q\oK^{n-q}\chi\big)|_{H^\ell(K)}\lesssim
   \hK^{m+1-\ell}d^{-(m+1)}
    \sum_{j=0}^{(m+1-q)\vee 0} d^{j}|\chi|_{H^j(K)},
\end{align}
 and the estimate \eqref{eq:SA1} follows immediately
 from \eqref{eq:Splitting} and \eqref{eq:P3}.
\par
 In the case where $n\geq m+1$, we have
   $$n-q\geq [(m+1-q)\vee0]$$
 and hence
   $$n-q-j\geq0$$
 for any index $j$ in the range of the summation
 on the right-hand side of \eqref{eq:P2}.  Therefore we
  can replace $\oK^{n-q}$ by $\oK^j$ inside the summation
 to obtain, through Lemma~\ref{lem:Binomial},
\begin{align}\label{eq:P4}
   &|\diff^q\oK^{n-q}\chi
  -\PK\big(\diff^q\oK^{n-q}\chi\big)|_{H^\ell(K)}\notag\\
   &\hspace{120pt}
   \lesssim
   \hK^{m+1-\ell}d^{-(m+1)}
    \sum_{j=0}^{(m+1-q)\vee 0} d^{j}|\oK^j\chi|_{H^j(K)}\notag\\
     &\hspace{120pt}
     \lesssim
   \hK^{m+1-\ell}d^{-(m+1)}
    \sum_{j=0}^{(m+1-q)\vee 0} d^{j}|\om^j\chi|_{H^j(K)}.
\end{align}
 The estimate \eqref{eq:SA2} follows immediately
 from \eqref{eq:Splitting} and \eqref{eq:P4}.


\begin{thebibliography}{10}

\bibitem{ADAMS:2003:Sobolev}
R.A. Adams and J.J.F. Fournier.
\newblock {\em {Sobolev Spaces $($Second Edition$)$}}.
\newblock Academic Press, Amsterdam, 2003.

\bibitem{AL:1995:Stokes}
D.N. Arnold and X.B. Liu.
\newblock Local error estimates for finite element discretizations of the
  {S}tokes equations.
\newblock {\em RAIRO Mod\'{e}l. Math. Anal. Num\'{e}r.}, 29:367--389, 1995.

\bibitem{BH:1970:Lemma}
J.H. Bramble and S.R. Hilbert.
\newblock {Estimation of linear functionals on Sobolev spaces with applications
  to Fourier transforms and spline interpolation}.
\newblock {\em SIAM J. Numer. Anal.}, 7:113--124, 1970.

\bibitem{BNS:1975:MaxNorm}
J.H. Bramble, J.A. Nitsche, and A.H. Schatz.
\newblock Maximum-norm interior estimates for {R}itz-{G}alerkin methods.
\newblock {\em Math. Comp.}, 29:677--688, 1975.

\bibitem{BScott:2008:FEM}
S.C. Brenner and L.R. Scott.
\newblock {\em {The Mathematical Theory of Finite Element Methods $($Third
  Edition$)$}}.
\newblock Springer-Verlag, New York, 2008.

\bibitem{Ciarlet:1978:FEM}
P.G. Ciarlet.
\newblock {\em {The Finite Element Method for Elliptic Problems}}.
\newblock North-Holland, Amsterdam, 1978.

\bibitem{DGS:2011:Local}
A.~Demlow, J.~Guzm\'{a}n, and A.H. Schatz.
\newblock Local energy estimates for the finite element method on sharply
  varying grids.
\newblock {\em Math. Comp.}, 80:1--9, 2011.

\bibitem{Leykekhman:2019:C0IP}
D.~Leykekhman.
\newblock {Pointwise error estimates for $C^0$ interior penalty approximation
  of biharmonic problems}.
\newblock {\em preprint}, 2019.
\begin{verbatim}
(https://www2.math.uconn.edu/~leykekhman/publications.html)
\end{verbatim}

\bibitem{NS:1974:SA}
J.A. Nitsche and A.H. Schatz.
\newblock {Interior estimates for Ritz-Galerkin methods}.
\newblock {\em Math. Comp.}, 28:937---958, 1974.

\bibitem{Wahlbin:2000:Handbook}
L.B. Wahlbin.
\newblock {Local Behavior in Finite Element Methods}.
\newblock In P.G. Ciarlet and J.L. Lions, editors, {\em Handbook of Numerical
  Analysis, II}, pages 353--522. North-Holland, Amsterdam, 1991.

\end{thebibliography}
\end{document}